\documentclass[12pt]{article}

\textheight=8.5in \topmargin=-0.3in \everymath{\displaystyle}
\usepackage{graphics}
\usepackage{amsmath}
\usepackage{amsfonts}
\usepackage{amssymb}
\usepackage{graphicx}
\input xy
\xyoption{all}

\usepackage{graphicx,amssymb,amsfonts,latexsym,amsmath,amsthm}

\usepackage{amsfonts}
\usepackage{amssymb}
\usepackage{eucal}

\newtheorem{theorem}{Theorem}[section]
\newtheorem{lemma}[theorem]{Lemma}
\newtheorem{proposition}[theorem]{Proposition}
\newtheorem{definition}[theorem]{Definition}

\def\R{{ \mathbb{R}}}

\everymath{\displaystyle}
\newcommand{\RE} {{\rm I \kern-2.8pt R} }

\newcommand{\beasnum}{\begin{eqnarray}}
\newcommand{\eeasnum}{\end{eqnarray}}
\newcommand{\beas}{\begin{eqnarray*}}
\newcommand{\eeas}{\end{eqnarray*}}
\begin{document}

\begin{center}
{\Large {\bf Evolutionary Game Theory on Measure Spaces: Well-Posedness}} {\large
\textbf{\bigskip}}

\vspace{0.1in}

John Cleveland$^{\dag}$ and Azmy S. Ackleh$^{\ddag}$

\vspace{0.1in}

$^\dag$Department of Mathematics\\
Penn State University\\
University Park, State College, PA 16802\\

\vspace{0.1in}

$^\ddag$Department of Mathematics\\
University of Louisiana at Lafayette\\
Lafayette, Louisiana 70504-1010\\

\vspace{0.1in}
\end{center}

\baselineskip = 18 pt

\begin{abstract}
An attempt is made to find a comprehensive mathematical framework in
which to investigate the problems of well-posedness and asymptotic
analysis for fully nonlinear evolutionary
game theoretic models. The model should be rich enough to include all classical nonlinearities, e.g.,  Beverton-Holt or Ricker type. For several such models formulated on the
space of integrable functions, it is known that as the variance of
the payoff kernel becomes small the solution converges in the long
term to a Dirac measure centered at the fittest strategy; thus the
limit of the solution is not in the state space of integrable
functions. Starting with the replicator-mutator equation and a
generalized logistic equation as bases, a general model is formulated as a
dynamical system on the state space of finite signed measures.
Well-posedness is established, and then it is shown that by choosing
appropriate payoff kernels this model includes all classical density
models, both selection and mutation, and discrete and continuous strategy (trait) spaces.  \\


\end{abstract}

\section{ Introduction}

Evolutionary game theory (EGT) is the creation and study of  mathematical models that describe how the strategy profile
in games change over time due to mutation and selection
(replication). In this paper we address the
problem of finding a comprehensive mathematical framework suitable
for studying the problems of well-posedness and long-term solution behavior for fully nonlinear
evolutionary game theoretic models. We form a unified theory for
evolutionary game theory as a dynamical system on the state space of
finite signed Borel measures under the weak star topology.
 In this theory, we unify the discrete and continuous strategy (trait) spaces
and the pure replicator and replicator-mutator dynamics under one model.

 A natural question to ask is why the formulation of a dynamical system on the state space of finite signed Borel measures under the weak star topology? Why isn't the existing mathematical machinery adequate? The next two examples will illustrate the need for such a formulation. First, we consider the following EGT model of generalized logistic growth with pure selection (i.e., strategies
replicate themselves exactly and no mutation occurs)  which was developed and
analyzed in \cite{AMFH}:
\begin{equation}
 \frac{d}{dt} x(t,q) = x(t,q) (
q_1 -q_2 X(t)), \label{logiseq}\end{equation}
 where $X(t) = \int_Q x(t,q) dq$ is the total population, $Q \subset \text{int}(\mathbb{R}_+^2)$ is compact
 and the state space is the set of continuous real valued functions
 $C(Q)$. Each $ q=(q_1, q_2) \in Q$ is a two tuple where $q_1$ is an
 intrinsic replication rate and $q_2$ is an intrinsic mortality
 rate. The solution to this model converges to a Dirac
 mass centered at the fittest $q$-class. This is the class with the highest birth to death ratio
 $\frac{q_1}{q_2}$,
 and this convergence is in a topology called $weak^* $
  (point wise convergence of functions) \cite{AMFH}. However, this Dirac limit is not in the
  state space as it is not a continuous function. It is a measure.  Thus, under this formulation one cannot treat this Dirac mass
  as an equilibrium (a constant) solution and hence the study of linear stability analysis is not possible. Other examples
  for models developed on classical state spaces such as $L^1(X,\mu)$ that demonstrate the emergence of Dirac measures in the asymptotic limit from smooth initial densities are given in \cite{AFT,AMFH,calsina,CALCAD,GVA,GR1,P,GR2}.
  In particular, how the measures arise naturally in a biological and adaptive dynamics environment is illustrated quite well in \cite[chpt.2]{P}.
   These examples show that the chosen state space for formulating such selection-mutation models must \textbf{contain} densities and Dirac masses and the topology used must \textbf{contain the ability to demonstrate convergence} of densities to Dirac masses.

The first example above assumes a continuous strategy space $Q$ and hence the model solution is sought among density functions denoted by $x(t,q)$. Our second example, is the classic discrete EGT model known as
the replicator-mutator equation (in this model the strategy space is assumed to be discrete). In \cite[pg. 273]{Nowak} it is
given as:
\begin{equation}\label{RM}
 \dot{x_i}= \sum _{j=1}^n x_jf_j(\overrightarrow{x})Q_{ij}
-\phi(\overrightarrow{x})x_i  \end{equation} where $\vec{x} =( x_1,
x_2, ..., x_n) $ is a vector consisting of $n$ classes each of size
$x_i$, and $Q_{ij}$ is the payoff kernel, i.e., $Q_{ij}$ is the
proportion of the $j$-class that mutates into the $i$-class. Lastly
$\phi =\sum_{j=1}^{n} f_jx_j $ is a weighted (average) fitness. The
author states that the language equation (replicator-mutator
equation) is a unifying description of deterministic evolutionary
dynamics. He further states that the replicator-mutator equation is
used to describe the dynamics of complex adaptive systems in
population dynamics, biochemistry and models of language
acquisition.

  Under the new formulation  on the space of measures we present here, the above examples are special cases of a more general measure-valued model.
  In particular, with the discrete model if we allow the fitness functions $f_j$ to be density dependent then this model can be obtained by choosing the proper initial condition composed of a linear combination
  of Dirac masses and the proper replication-mutation kernel which is also composed of a linear combination of Dirac
  masses. The example of the pure selection density model given in \eqref{logiseq} can be realized from the measure-valued model by choosing an
  absolutely continuous initial measure and a continuous family of
  Dirac measures for the selection-mutation kernel (which represents the pure replication case).
  Thus, these density and discrete models can be unified under this formulation. Furthermore, our new theory combines
  both the pure replicator and replicator-mutator dynamics in a continuous manner. By this we mean that as the mutations get
  smaller and smaller the replicator-mutator model will approach the pure replicator
  model. This is possible because our mutation kernels are allowed
  to be (family of) measures as well. This presents a serious difficulty in the
  analysis which requires the development of some technical tools in studying the
  well-posedness of the new model.

 Many researchers have recently devoted their attention to
the study of such EGT models (e.g.
\cite{AFT,AMFH,calsina,CALCAD,Hof,Hof2, MagWeb,SES}). To date almost
all EGT models are formulated as {\it density} models
\cite{AMFH,calsina,CALCAD, MagWeb,SES} with {\it linear} mutation
term. There are several formulations of pure selection or replicator
equation dynamics on measure spaces \cite{AFT,Bomze,CressHof}. The
recent formulations of selection-mutation balance equations on the
probability measures by \cite{EMP,KKO} are novel constructions.
These models describe the aging of an infinite population as a
process of accumulation of mutations in a genotype. The dynamical
equation which describes the system is of Kimura-Maruyama type. Thus
far in selection-mutation studies the mutation process has been
modeled using two different approaches: (1) A diffusion type
operator \cite{GVA,SES}; (2) An integral type operator that makes use of
a mutation kernel \cite{AFT,calsina,CALCAD,EMP, KKO}. Here we focus
on the second approach for modeling mutation.

Perhaps the work most related to the one presented here is that in \cite{AFT}. In that paper,
the authors considered a {\it pure} selection model with density dependent birth and mortality
function and a 2-dimensional trait space on the space of finite signed measures. They discussed
 existence-uniqueness of solutions and studied the long term behavior of the model.
Here, we generalize the results in that paper in several directions.
Most salient is the fact that the present paper is one in evolutionary game
theory, hence the applications are possibly other than population
biology. In particular, in the present paper we construct a (measure valued) EGT model. This is
an ordered triple $(Q,\mu,F)$ subject to:
\begin{equation}\label{mconstraint}\frac{d}{dt}\mu(t)(E)=F(\mu(t)(Q))(E), \text{ for every}
~~E \in \mathcal{B}(Q). \end{equation} Here $Q$ is the strategy
(metric) space, $\mathcal{B}(Q)$ are the Borel sets on $Q$, $\mu(t)$
is a time dependent family of finite signed Borel measures on $Q$
and $F$ is a density dependent vector field such that $\mu$ and $F$
satisfy equation \eqref{mconstraint}. The main contributions of the present work
are as follows:  (1) we establish well-posedness of the new
measure-valued dynamical system; (2) we are able to combine models
that consider both discrete and continuous parameter spaces under
this formulation; no separate machinery is needed for each; (3) we
are able to include both selection and mutation in one model because
our setup allows for choosing the mutation to be a family of
measures; (4) unlike the linear mutation term commonly used in the literature, we allow for nonlinear (density dependent) mutation
term that contain all classical nonlinearities, e.g., Ricker,
Beverton-Holt, Logistic; (5) unlike the one or two dimensional strategy spaces used in the literature,  we allow for
a strategy space $Q$ that is possibly infinite dimensional. In particular, we assume that $Q$ is a compact complete separable metric space, i.e., a compact Polish space;

This paper is organized as follows. In section 2 we demonstrate how
to proceed from a density model to a measure valued one and we
formulate the model on the (natural) space of measures. In section 3
we establish the well-posedness of this model. In section 4 we
demonstrate how this model encompasses the discrete, continuous
replicator-mutator and species and quasi-species models.  In section 5 we provide concluding
remarks.

 \section{ \bf From Densities to Measures}
We begin by giving a definition of a dynamical system that will be
used throughout this paper.
\begin{definition} If $\mathfrak{T}$, $\Gamma$ are topological spaces, then a dynamical system on $\mathfrak{T}$ is the tuple
 $( \mathfrak{T}, \Gamma, \varphi)$ where, $ \varphi: {\mathbb {R}_+}\times \mathfrak{T} \times \Gamma \to \mathfrak{T}$
is such that the following hold:
\end{definition}

\begin{description}
\item i.  For all $(u, \gamma) \in \mathfrak{T}\times \Gamma$, $ \varphi(\cdot;u, \gamma)$ is continuous.
\item ii. For all $(u, \gamma) \in \mathfrak{T} \times \Gamma$,  $ \varphi( 0; u, \gamma) =u$.
\item iii. For all $ \theta_1, \theta_2, u, \gamma $, $\varphi(\theta_1+ \theta_2; u, \gamma) =
\varphi(\theta_2;\varphi(\theta_1, u, \gamma) , \gamma).$
\item iv. If $\varphi$ is a continuous mapping then $\varphi$ is called a
continuous dynamical system.
\end{description}
There is a natural equivalence between dynamical systems and initial
value problems. Given an initial value problem (IVP), the solution
as a function of the parameter, initial condition and starting time
generate a dynamical system \cite{ChellHadd}. Our dynamical system
will be the one resulting from the solution of an IVP. To this end
our initial modeling point is to take as the strategy space $Q$ a compact subset of ${\rm int}(\mathbb{R}^n_{+})$ (the interior of the positive cone of
$\mathbb{R}^n$).  and to consider the following density IVP:
\begin{equation} \left\{
\begin{array}{l}
 \frac{d}{dt}x(t,q) =
\underbrace{\int_Q f_1(X(t),\hat q) p(q,\hat q) x(t,\hat q) d \hat
q}_{\mbox{Birth term}} - \underbrace{f_2(X(t),q))
x(t,q)}_{\mbox{Mortality term}}\\
 x(0,q)= x_{0}.
 \label{dsmm} \end{array}
\right.\end{equation} Here, $X(t) = \int_Q
x(t,q) dq$ is the total population, $f_1(X,\hat q)$ represents the
density-dependent replication rate per $\hat  q $ individual, while $f_2(X,q)$ represents the density-dependent mortality rate per $q$ individual.
 The probability density function $p(q,\hat q)$ is the
selection-mutation kernel. That is, $p(q, \hat q) d q$ represents
the probability that an individual of type $\hat q$ replicates an
individual of type $q$ or the proportion of $\hat q$'s offspring
that belong to the $dq$ ball. Hence, $f_1(X(t),\hat q)p(q, \hat q)
dq$ is the offspring of $\hat q$ in the $dq$ ball and $
f_1(X(t),\hat q) p(q,\hat q)dq x(t,\hat q) d \hat q$ is the total
replication of the $d\hat q$ ball into the $dq$ ball. Summing
(integrating) over all $ d\hat q$ balls results in the replication
term. Clearly $f_2(X(t),q) x(t,q)dq$ represents the mortality in the
$dq$ ball. The difference between birth and death in the $dq$ ball
gives the net rate of change of the individuals in the $dq$ ball,
i.e., $\frac{d}{dt}x(t,q)dq .$ Dividing by $dq$ we get \eqref{dsmm}.

We  point out that \textbf{formally}, if we let $p(q,\hat q)
=\delta_{\hat q}(q)=\delta_{q}(\hat q)$ (the delta function is even)
in \eqref{dsmm} then we obtain the following pure selection
(density) model
\begin{equation} \left \{ \begin{array}{l}
\displaystyle \frac{d}{dt}x(t,q) =   x(t,q) (f_1(X(t),q) -
f_2(X(t),q))\\
 x(0,q)= x_{0},
 \end{array} \right .
 \label{pureselection}
\end{equation}
of which equation (1) in \cite{AFT} is a special case. Indeed if
$p(q,\hat q)dq = dq\delta_{\hat q} ( q)$ then this means that the
proportion of $\hat q$'s offspring in the $dq$ ball is zero unless
$q= \hat q$ in which case this proportion is $dq,$ i.e., individuals
of type $\hat q$ only give birth to individuals of type $\hat q$.

Integrating both sides of \eqref{dsmm} over a Borel set $E \subset Q$, we obtain
$$ \int_E \frac{d}{dt}x(t,q)dq =  \int_E
\bigl[\int_Q f_1(X(t),\hat q) p(q,\hat q) x(t,\hat q) d \hat q  -
f_2(X(t),q) x(t,q)\bigr]dq .$$ Changing order of integration we get
$$ \begin{array}{lll}
\int_E \frac{d}{dt}x(t,q)dq &=& \int_Q f_1(X(t),\hat q)\bigl[\int_E
p(q,\hat q)dq\bigr]
x(t,\hat q) d \hat q - \int_E f_2(X(t),q) x(t,q)dq\\
&=& \int_Q f_1(X(t),\hat q)\gamma(\hat q) (E) x(t,\hat q) d \hat q -
\int_E f_2(X(t),q) x(t,q)dq,
\end{array}
$$
where $\gamma(\hat q)(E)=\int_E p(q, \hat q) dq$ is the proportion
of $\hat q$ 's offspring in the Borel set $E$.

This yields the following measure valued dynamical system:
\begin {equation} \left\{\begin{array}{ll}\label{M1}
 \displaystyle \frac{d}{dt}{\mu}(t; u, \gamma)(E) = \int_Q {f}_1(\mu(t)(Q), \hat q) \gamma(\hat q)(E)d\mu(t)(\hat q)\\
\hspace{1.2 in} - \displaystyle \int_E {f}_{2}(\mu(t)(Q),\hat q)
d\mu(t)(\hat q) =  {F} (\mu, \gamma)(E) \\
\mu(0; u,\gamma)=u.
\end{array}\right.\end{equation}

\section{Well-Posedness of Measure-Valued Dynamics }\label{wp}
In this section we focus on the well-posedness of the model \eqref{M1}.  This requires
setting up some notation and notions and establishing several lemmas
and propositions. To this end, throughout Section \ref{wp} the strategy space $(Q,d)$ will be a compact
complete separable metric space otherwise known as a compact Polish
space. The reader may think of a compact Riemannian manifold or a
compact subset of ${\rm int}(\mathbb{R}_+^n).$

\subsection{Birth and Mortality Rates}
Concerning the birth and mortality densities $f_1$ and $f_2$ we make assumptions similar
to those used in \cite{AFT}:
\begin{itemize}
\item[(A1)] $f_1: \mathbb{R}_+ \times Q \rightarrow \mathbb{R}_+$ is locally Lipschitz
continuous in $X$ uniformly with respect to $q$, nonnegative, and nonincreasing  on $\mathbb{R}_+$ in $X$
and continuous in $q$.
\item[(A2)] $f_2: \mathbb{R}_+ \times Q \rightarrow \mathbb{R}_+$ is locally Lipschitz continuous in $X$ uniformly with respect to $q$,
nonnegative, nondecreasing on $\mathbb{R}_+$ in $X$, continuous in
$q$ and $ \inf_{q \in Q} {f_{2}(0, q)} =\varpi
>0 $. (This means that there is some inherent mortality not density
related)
\end{itemize}
These assumptions are of sufficient generality to capture many nonlinearities of classical population dynamics including Ricker,
Beverton-Holt, and Logistic (e.g.,  see \cite{AFT}).

 \subsection{Technical Preliminaries for Measure Valued Formulation
 }\label{subsection1}

\subsubsection{Important Notation and Technical Definitions}

 We will use the symbol ${\cal M}$ to denote the set of
finite signed Borel measures when we wish to view it as a Riesz
space \cite{AliBord} and $\mathcal{M}_+$ will denote its positive
cone. If the total variation norm is denoted $|\cdot|_V$, then
$\mathcal{M}_V$ will denote the Banach space of the finite signed
measures with the total variation norm. Definition \ref {duality} in
the Appendix tells us that the duality $<C(Q),\mathcal{M}>$ given by
$ <f,\mu> \mapsto \int_Q f(q) d \mu $ generates a $weak^*$ topology
on $\mathcal{M}$ which we denote as $\mathcal{M}_w$, i.e., the
locally convex TVS (topological vector space) $(\mathcal{M},
\sigma(\mathcal{M},C(Q)).$ If $S \subseteq \mathcal{M}$, $S_w$
denotes the same set under the $weak^*$ topology and $S_V$ the same
set under total variation. If no topology is indicated then $S$ is
simply a subset of the Riesz space of ordered measures. Also $S_+ =
S \cap \mathcal{M}_+$. Let $\mathcal P_w $ denote the probability
measures under the $weak^*$ topology and $ C^{po} =
C(Q,\mathcal{P}_w(Q))$, the continuous functions on $Q$ with the
topology of uniform convergence.

Note that the EGT model we study here is a dynamical system arising
from an ODE. A common method used to establish existence and
uniqueness of solutions to such dynamical systems is to apply a
contraction mapping argument to a suitably chosen complete metric space. Indeed, this is the method we adopt
here.

 To this end
if $a, b>0 $ and $ \mu_{0} \in
\mathcal{M}_{+}$ are given, let $ I_b({0})$ be the interval
$[0,b)$, and $\overline{B_{a}(\mu_{0})}$ be
the closed total variation ball of radius $a$ around $\mu_{0}.$
Since the space of finite signed measures $ \mathcal{M}_V$ under
total variation norm is a Banach Space,  if $\cal X$ is any set then
the bounded maps from $\cal X$ into $\mathcal{M}_V$ under the sup
norm, i.e., $\|f\|_S= \sup_{x \in {\cal X} } | f(x)|_ {V}$ is
another Banach space denoted $ \cal{BM}(\cal X) := ({\cal BM} ({\cal
X}), \|\cdot\|_S)$.  $\cal{BM}(\cal X)$ is the
space in which we are {\it always} working and should be kept in
mind when we begin the fixed point argument as there are several
topologies being used.
For our dynamical system purposes we are interested in the set $
{\cal X} = \overline{I_b({0})} \times
(\overline{B_{a,+}(\mu_{0})})_{w} \times C^{po}. $ Let's denote by
      $ {\cal{C}} \Bigl( \overline{I_b({0})} \times
(\overline{B_{a,+}(\mu_{0})})_w \times C^{po} ;
(\overline{B_{2a}(\mu_{0})})_w \Bigr)$ the closed subcollection of continuous maps into $(\overline{B_{2a}(\mu_{0})})_w.$
Then it is an exercise to show that $(M(a,b),\|\cdot\|_S)$ where
$$ \begin{array}{l}  M(a,b)= \{ \alpha \in {\cal{BM}}({\cal X})|~~
\alpha \in {\cal{C}} \Bigl( \overline{I_b({0})} \times
(\overline{B_{a,+}(\mu_{0})})_w \times C^{po} ;
(\overline{B_{2a}(\mu_{0})})_w \Bigr) ,  \\
\hspace{4 in} \alpha \ge 0, \alpha (0;u,\gamma) =u \}
\end{array}
$$ is a nonempty closed metric subspace of the
complete metric space ${\cal{BM}} \Bigl(
 \overline{I_b({0})} \times
\overline{B_{a,+}(\mu_{0})}  \times C^{po}\Bigl).$

We will let $\vec{\textbf{0}}$ denote the zero measure, \textbf{1}
denote the constant function one (from $Q$ to $\mathbb{R}$), and if $\alpha \in M(a,b)$, then
we will at times write $\alpha(t)$ for $\alpha(t;u,\gamma)$ when we
are keeping $ u,\gamma$ fixed.

 \subsubsection{Families of Measures and Mutation Kernels
 $\Bigl(
\int_Q f_1(X, \hat q) \gamma(\hat q) d\mu( \hat q), \\ \int_{T\times
Q } f_1(X(s),\hat q) \overline{\gamma}_{s,t,
\alpha(\cdot;u,\gamma)}(\hat q) d\mu( \hat q) \times ds
\Bigr) $ } \label{familiesofmeasures}
 In order to understand this section we must first understand all of the duals that we will be using. As they can be confusing. Given a vector space $V$ or more generally a Riesz space, one
automatically has an algebraic dual denoted $V^{\sharp}$. If $V$ is
also a topological vector space, then there is the continuous dual
denoted $V'$ with the relation $V' \subseteq V^{\sharp}$. Since $V'$
is also a vector space we can form ${V'}^{\sharp}$ which has the
relation $V \subseteq {V^{\sharp}}^{\sharp}\subseteq {V'}^{\sharp}$. The first $ \subseteq$ is actually the natural algebraic monomorphism   $v \mapsto \delta_{v}$.
So given $v\in V$ there are three ways to view this element given by
each inclusion. We shall have occasion to use this fact when
defining our mutation term. \\

A measure is both a countably additive set function
and also a continuous linear functional on $C(Q)$ \cite{NB2}. For
example, if $\nu$ is a measure
$$\nu(\textbf{1}) =\nu(Q) =\int_Q d\nu .$$ Each
view is useful in its own right. For example, if one wishes to model
the sizes of populations then speaking of the``measure" of a Borel
set intuitively has the meaning size of population. Speaking of the
value of a linear functional on a continuous function is less
intuitive biologically. However, for mathematical purposes at times
the linear functional viewpoint is more beneficial. So in our proofs we will use the functional definition, however in biological explanations we will use the set function approach.\\

We are all familiar with point masses and absolutely continuous
measures. However, in the formulation of this model we come upon a
novel type of measure. This measure is defined as the integral of a
family of measures. If $T$ is a closed interval of $\mathbb{R}_+$,
then $T\times Q$ is compact. If $ \gamma \in C^{po}$, $\alpha \in
M(a,b)$, then define
$${\overline{\gamma}}_{s,t, \alpha(\cdot;u,\gamma)}(\hat q)(E) =
 \int_{E}e^{-\int_{s}^{t} f_{2}(\alpha(\tau)(Q),q)d\tau}d\gamma(\hat q)(q). $$
From a biological point of view  $\overline{\gamma}_{s,t,
\varphi(\cdot;u,\gamma)}(\hat q)(E)$ is the
 net proportion of $\hat q $'s offspring that belong to $E$ from time $s$ to time $t$.
 Since $f_{1}(\varphi(s;u, \gamma)(Q),\hat q) \mu(d\hat q)$ is the number of
offspring produced by a $d\hat q $ ball, $f_{1}(\varphi(s;u, \gamma)(Q),\hat
q)\overline{{\gamma}}_{s,t,\varphi(\cdot;u, \gamma)}(\hat q)(E)\mu(d\hat q) $ is
the total contribution of the $d\hat q $ ball to the Borel set $E$ by total
new recruits from time $s$ to $t$.

If $f_1$ is bounded and $\gamma \in C^{po} $, then we wish to
consider two mappings: (1) for each $X$ the mapping $\hat q \mapsto
f_1(X, \hat q) \gamma(\hat q)$; (2) $(s,\hat q)\mapsto f_1(X(s),
\hat q){\overline{\gamma}}_{s,t, \alpha(\cdot; u,\gamma)}(\hat q)$.
They are both weakly continuous mappings with compact support that
map into a complete convex subset of the locally convex space
$\mathcal{M}_w$. So if $\mu \in \mathcal{M}$, then $\int_Q f_1(X,
\hat q) \gamma(\hat q) d\mu( \hat q) $ and $\int_{T\times Q }
f_1(X(s),\hat q) \overline{\gamma}_{s,t, \alpha(\cdot;
u,\gamma)}(\hat q) (d\mu( \hat q) \times ds ) $ exists and are also
elements of $\mathcal{M}$ by Theorem \ref{int} in the Appendix. Let
us be more clear. These two integrals are elements of $\mathcal{M}$
in the following sense. Let $~ \widehat{\cdot}~ $ denote the
canonical algebraic imbedding $~ \widehat{\cdot} : \mathcal{M}
\hookrightarrow (\mathcal{M}^{\sharp})^{\sharp}$ given by
 $ \widehat{\nu} (f) = f(\nu)$ or $\widehat{\nu} = \delta_{\nu}$, where $ \delta_{\nu} (f) = f(\nu)$
for $ f \in \mathcal{M}{^\sharp}$ is the evaluation homomorphism.
Since $(\mathcal{M}_w)' \subseteq \mathcal{M}{^\sharp}$,
$(\mathcal{M}{^\sharp})^\sharp \subseteq (\mathcal{M}_w')^\sharp .$
So $\nu \mapsto \delta_{\nu}$ actually algebraically imbeds $
\mathcal{M} \hookrightarrow (\mathcal{M}_w')^{\sharp}.$ So viewing
$\nu$ as the algebraic linear functional $\delta_{\nu}$ is what we
mean. More to the point, let us consider only the first integral
 $\int_Q f_1(X, \hat q) \gamma(\hat q) d\mu( \hat q)
$, since the second can be understood similarly. By Theorem
\ref{int} and the above discussion $\int_Q f_1(X, \hat q)
\gamma(\hat q) d\mu( \hat q) $ is an element $\delta_{\nu} \in
(\mathcal{M}_w')^{\sharp}$ and by Definition \ref{integral}
 $$ z'(\nu)= \delta_{\nu}(z')=\mu(z'(f_1\gamma))\quad \text{ where } \quad z' \in
(\mathcal{M}_w)'.$$ Since $<C(Q),\mathcal{M}_V> $ is a duality by
Theorem \ref{Duality}, for $z' \in (\mathcal{M}_w)'$ there is a
unique $z \in C(Q)$ such that $z'(\mu)=<z,\mu>$ for all $\mu \in
\mathcal{M}$.
 Hence, \begin{equation}\label{I}
 <z',\hat \nu>=<z',\delta_{\nu}>=\delta_{\nu}(z')=z'(\nu)=<z,\nu>, \quad \text{for all} ~ z'
\in (\mathcal{M}_w)'.
\end{equation}
So if $\int_Q f_1(X, \hat q) \gamma(\hat q) d\mu( \hat q)
=\delta_{\nu}=\widehat{\nu} $, then we define $\int_Q f_1(X, \hat q)
\gamma(\hat q) d\mu( \hat q)$ to be the measure $\nu$ which behaves
as in \eqref{I}. Similarly for $\int_{[0,T]\times Q } f_1(X(s),\hat
q) \overline{\gamma}_{s,t, \alpha(\cdot;u,\gamma)}(\hat q) (d\mu(
\hat q) \times ds )$.

If $f$ is continuous, the measure $ E \mapsto \int_E f(q) d\mu(q)$ as a functional has the action: $ z \mapsto \int_Q z(q)f(q) d\mu(q) $ for $z \in C(Q)$. For the remainder of this section we will denote such a functional as $ <\int f(q) d\mu(q), \cdot>$.
Before we end this section we will draw a connection between
the continuous functional and set function aspects of these families
of measures.

\begin{theorem}\label{Pos} Let $\mu \in \mathcal{M}_+ $. If $f: Q\rightarrow \mathcal{M}_w$ is continuous and bounded in total variation,
then $$\Bigl(\int_Q f(\hat q) d\mu \Bigr )(E) = \int_Q f(\hat q)(E)
d\mu$$ for every Borel set E.
\end{theorem}
\begin{proof} We give a sketch of the proof and refer the reader
to \cite{AliBord,NB2,Royden} for background definitions and details.
Since $Q$ is a metric space, it is outer normal, hence outer regular
\cite[pg. 379]{AliBord}. Thus, the value of a finite signed measure
is completely known once it is known on open sets. To this end let
$\nu_2(E)= \int_Q f(\hat q)(E) d\mu$. Then it is an elementary
exercise to demonstrate that $\nu_2 $ is a finite signed measure
\cite{Royden}. Using Theorem \ref{int} and the analysis before this
theorem $\nu_1=\int_Q f(\hat q) d\mu \in \mathcal{M}$.

We will show that $\nu_1 = \nu_2$ on open sets. By definition, since
the characteristic functions of open sets are lower semi-continuous,
if $G$ is open and $\varphi_G$ is its characteristic function, then
$$
\begin{array}{ll}
\nu_1(G) ~:= & \nu_1^*(\varphi_G) = \sup_{h \in C(Q), h \leq
\varphi_G}\nu_1(h)=\sup_{h \in C(Q), h \leq \varphi_G} \int_Q f(\hat
q)(h) d\mu \\ \\ & = \int_Q \sup_{h \in C(Q), h \leq \varphi_G}
f(\hat q)(h) d\mu = \int_Q  (f(\hat
q))^*\underline{\underline{}}(\varphi_G) d\mu =:\nu_2(G)  .
\end{array}$$ \end{proof}

\subsection{ Main Well-Posedness Theorem}
The following is the main theorem of this section.
\begin{theorem}\label{main}  There exists a continuous dynamical system $({\cal M}_{+,w}, C^{po},\varphi)$ where $\varphi:
{\mathbb {R}_+} \times{\cal M} _{+,w} \times C^{po} \to {\cal M}_{+,w}
  $ satisfies the following:
\begin{enumerate}
\item  For fixed $ u, \gamma $, the mapping $t \mapsto \varphi(t;
u,\gamma)$ is continuously differentiable in total variation, i.e.,
$\varphi(\cdot, u, \gamma): {\mathbb {R}_+} \to {\cal M}_{V,+}$.
 \item For fixed $ u, \gamma $, the mapping $t \mapsto \varphi(t; u,\gamma)$ is the unique \emph{solution} to

 \begin {equation} \left\{\begin{array}{ll}\label{M}
 \displaystyle \frac{d}{dt}{\mu}(t)(E) = \int_Q {f}_1(\mu(t)(Q), \hat q) \gamma(\hat q)(E)d\mu(t)(\hat q)\\
\hspace{1.2 in} - \displaystyle \int_E {f}_{2}(\mu(t)(Q),\hat q)
d\mu(t)(\hat q) =  {F} (\mu, \gamma)(E) \\
\mu(0)=u.
\end{array}\right.\end{equation}
\end{enumerate}
\end{theorem}

We now establish a few results that are needed to prove
Theorem \ref{main}.
 \subsubsection{Local Existence and Uniqueness of
Dynamical System}\label{subsection2}

First let $ \mu_{0} \in
\mathcal{M}_+$ and $a > 0$ be fixed. As it stands $F(\mu,\gamma)$ as
defined in \eqref{M} need not be a finite signed measure at all. If
$\mu(t)(Q)$ is ever negative, then $F(\mu(t),\gamma)$ is not
defined. So we modify $F$ as follows:  Choose $\tilde K >  \mu_{0}(Q) +2a$.
 For $j=1,2$, extend $f_j$ to $\R \times Q$ by setting $\widetilde{f}_j(x,q) = f_j(0,q) $ for $x \le 0$
 and make the modification $ \widetilde{f}_j(x,q) = f_j(\tilde K,q) $ for $x \ge \tilde
 K$. Then $ \widetilde{f}_j: \R\times Q \to \R_+$ are  Lipschitz
continuous in the first variable and bounded with Lipschitz
constants $L_j$ and bounds $B_j$. Let $ \widetilde{F} (\mu, \gamma)
(E)$ be the redefined vector field obtained by replacing $f_j$ with
$\widetilde{f}_j$. The function $ \widetilde{F} (\mu, \gamma) $
is now a finite signed measure.

\begin{lemma} \label{LF}(Lipschitz F)  Let $\widetilde{F}$ be as above and let
$W \subseteq \mathcal{M}$ be bounded in total variation. Then for
every $\gamma \in C^{po}$ we have the following:
\begin{enumerate}
\item  There exists a continuous function $K_{\widetilde{F}} \geq 0,$ such that $|\widetilde{F}(\alpha, \gamma)|_V \leq
K_{\widetilde{F}}(|\alpha|_V) |\alpha| _V,$ $ \forall \alpha \in
\mathcal{M}$.
\item  $\widetilde{F}(\alpha, \gamma)$ is
bounded and uniformly Lipschitz continuous on $(W_{+})_V \times
C^{po}$ in $\alpha.$
\end{enumerate}

\end{lemma}
\begin{proof}
\begin{enumerate}
\item Define $K_{\widetilde{F}}(s) = B_{1} +B_{2} +(L_{1}+L_{2})s $. Since
$ F( \vec{0}, \gamma) =\vec{0} $, this follows from equation
\eqref{lipestimate} below.
\item Let $C_W$ be a
bound for $W$ in the norm topology, i.e., $|\mu|_V\leq C_W$ for $\mu
\in W$. We now prove uniform Lipschitz continuity in $\alpha $. The
boundedness trivially follows. Given $W$, notice that for all
$\alpha \in W$, $K_{\widetilde{F}}(|\alpha|_V) \leq
K_{\widetilde{F}}(C_W)$.
 If $\alpha$ and $\beta$ are finite signed measures, then
$d(\alpha)= d(\alpha - \beta + \beta)$.
 Hence, $$ \begin{array}{lll} \widetilde{F}(\alpha,\gamma) - \widetilde{F}(\beta,\gamma)&=& \int_Q
\gamma(\hat q)[ \widetilde{f}_{1}(\alpha(Q),\hat q) -
\widetilde{f}_{1}(\beta(Q),\hat q)]d\alpha(\hat q)\\
&& \quad  +
 \int_{Q}\widetilde{f}_{1}(\beta(Q),\hat
q)\gamma(\hat q)d(\alpha-\beta)(\hat q) \\
&&\quad - <\int[ \widetilde{f}_{2}(\alpha(Q),\hat q)-
\widetilde{f}_{2}(\beta(Q),\hat q)]d\alpha(\hat q), \cdot>
 \\
 && \quad -<\int \widetilde{f}_{2}(\beta(Q),\hat q)d(\alpha-\beta)(\hat
q), \cdot>, \end{array}$$ and
 $$ | \widetilde{F} (\alpha,\gamma) -
\widetilde{F}(\beta,\gamma)|_V \leqslant |\alpha|_V L_1|\alpha
-\beta|_V + B_1|\alpha - \beta|_V + |\alpha|_VL_2|\alpha - \beta|_V
+ B_2|\alpha -\beta|_V .$$  Thus,\begin{equation}
\label{lipestimate} | \widetilde{F} (\alpha,\gamma) -
\widetilde{F}(\beta,\gamma)|_V \leq
K_{\widetilde{F}}(|\alpha|_V)|\alpha-\beta|_V \leq
K_{\widetilde{F}}(C_W)|\alpha-\beta|_V .\end{equation}
\end{enumerate}
\end{proof}

\begin {lemma}\label{E}(Estimates) If $\alpha, \beta \in M(a,b) $, $ t_1, t_2 \in  \mathbb{R}_+$, $\mu_{0} \in  \mathcal{M} _+$
pick constants $ C_1, C_2 $ as follows:
 $C_1 =\mu_{0}(Q) + 2a, C_2 = L_1 + 2bL_2B_1.$  We have the following estimates:
\begin{enumerate}
\item $ {\biggl |\int_{[t_1,t_2]\times Q}
\widetilde{f}_1(\alpha(s)(Q), \hat q)
 \overline{\gamma}_{s,t, \alpha}(\hat q)d\alpha(s)ds\biggr |}
_{V} \leq  2bC_1B_1$.
\item $ |e^{-\int_{t_1}^{t_2} \widetilde{f}_2 (\alpha({\tau})(Q),q)d\tau} -
e^{-\int_{t_1}^{t_2} \widetilde{f}_2(\beta({\tau})(Q),q)d\tau}|\leq $
$\|(\alpha-\beta)\|_{S}L_2 2b$ for all $t_1, t_2 \in I_b({0})
$.

\item ${\biggl |\widetilde{f_1}(\alpha(s; u, \gamma)(Q), \hat
q){\overline{\gamma}}_{s,t,\alpha}( \hat q)-
\widetilde{f_1}(\beta(s; u, \gamma)(Q), \hat q) {\overline{\gamma}
}_{s,t,\beta}( \hat q)\biggr |}_{V} \leq C_2 \|\alpha -\beta\|_{S}.$
\end{enumerate}
\end{lemma}

\begin{proof} \begin{enumerate}
\item Recall from \cite[pg. 185]{AliBord} that $\|\nu\|_V= \sup_{\| f \|_\infty \le 1} |< \nu,f>|
$. Initially if $f \in C(Q)$, then we have by using Definition
\ref{integral} $$\begin{array}{lll} && {\biggl
|\int_{[t_1,t_2]\times Q} \widetilde{f}_1(\alpha(s)(Q), \hat q)
 \overline{\gamma}_{s,t, \alpha}(\hat q)d\alpha(s)ds\biggr |}
_{V} \\
&& \hspace{0.8 in} = \sup_{\|f \|_{\infty} \leq
1}|<\int_{[t_1,t_2]\times Q}
\widetilde{f}_1\overline{\gamma}_{s,t, \alpha}, f>| \\
&& \hspace{0.8in}  = \sup_{\|f \|_{\infty} \leq 1}|< \widehat{\Bigl
(\int_{[t_1,t_2]\times Q}
\widetilde{f}_1\overline{\gamma}_{s,t, \alpha}\Bigr )}, f'>| \\
&& \hspace{0.8 in} = \sup_{\|f \|_{\infty} \leq 1} \Bigl
|\int_{[t_1,t_2]\times Q}< \widetilde{f}_1\overline{\gamma}_{s,t,
\alpha}(\hat q), f'>
d\alpha(s) \times ds \Bigr | \\
&& \hspace{0.8in}  \leq 2bC_1 {\sup_{\|f \|_{\infty}\leq 1,(s, \hat
q)\in [t_1,t_2]\times Q} }
 |< \widetilde{f}_1 \overline{\gamma}_{s,t,\alpha}(\hat q),f'>|  \leq 2bC_1B_1, \end{array}$$
 since $$ \begin{array}{l} |< \widetilde{f}_1
 \overline{\gamma}_{s,t,\alpha}(\hat q),f'>| = | \int_{Q} f(q) d (\widetilde{f}_1 \overline{\gamma}_{s,t,\alpha}( \hat
 q))(q)| \\
 \qquad = |\int_{Q}f(q) \widetilde{f}_1(\alpha(s)(Q), \hat q)e^{-\int_{s}^{t} \widetilde{f}_{2}(\alpha(\tau)(Q),q)d\tau}d\gamma(\hat
 q)(q)| \leq B_1 .\end{array}$$

(see  subsection \ref{familiesofmeasures} for the notation $\widehat{\Bigl
(\int_{[t_1,t_2]\times Q}
\widetilde{f}_1\overline{\gamma}_{s,t, \alpha}\Bigr )}$ and $ f^{'}$).

\item There exists $\xi > 0$, such that
 $$ \begin{array}{lll} && \biggl| e^{-\int_{t_1}^{t_2}\tilde
f_{2}(\alpha(\tau)(Q),q)d\tau}-e^{-\int_{t_1}^{t_2}\tilde
f_{2}(\beta(\tau)(Q),q)d\tau}\biggr | \\
&& \quad = e^{-\xi}\biggl |\int_{t_1}^{t_2}
\Bigl[\widetilde{f}_2(\beta(\tau; u, \gamma)(Q),q) -
\widetilde{f}_2(\alpha(\tau; u,\gamma)(Q) , q) \Bigr]d\tau \biggr |
 \\
 && \quad \leq 2bL_2 \|\alpha - \beta \|_{S} .\end{array}$$

\item  For the third estimate we have:
$$\begin{array}{l} {\biggl |\widetilde{f_1}(\alpha(s; u, \gamma)(Q), \hat q){\overline{\gamma}}_{s,t,\alpha}( \hat
q)- \widetilde{f_1}(\beta(s; u, \gamma)(Q), \hat q)
{\overline{\gamma} }_{s,t,\beta}( \hat q)\biggr |}_V \\
\quad  \leq | \widetilde{f_1}(\alpha(s; u, \gamma)(Q), \hat q)-
\widetilde{f_1}(\beta(s; u, \gamma)(Q), \hat q){|\biggl
|{\overline{\gamma}}_{s,t,\alpha}( \hat q)\biggr |}_V \\
\qquad + \widetilde{f_1}(\beta(s; u, \gamma)(Q), \hat q) {\biggr
|{\overline{\gamma}}_{s,t,\alpha}( \hat q)-
{\overline{\gamma}}_{s,t,\beta}( \hat q)\biggr |}_V \\
\quad \leq L_1 |\alpha(s)- \beta(s)|(Q) + 2bL_2B_1\|\alpha - \beta
\|_S \leq (L_1 + 2bB_1 L_2 ) \|\alpha - \beta \|_{S} . \end{array}
$$
\end{enumerate}
\end{proof}

\begin {lemma}\label{FP}(Fixed Point) If $ \mu_{0} \in \mathcal{M}_+ $, let $a>0$,
 $C_1, C_2$ be as in Lemma \ref{E}, with $b$ such that $ (1-e^{-B_2b})\mu_{0}(Q) + 2B_1C_1b <
a $ and $ b < \min\{1, \frac{1}{2L_2C_1 + 2B_1 +2C_2C_1} \}$. Then $
S \colon M(a,b)\to M(a,b) $ given by
\begin{equation}\label{irep1} \begin{array}{l}
[S \alpha](t;u, \gamma) = <\int e^{-\int _{0}^{t}
\widetilde{f}_{2}(\alpha(\tau)(Q),q)d\tau}du(q),\cdot>\\
\hspace{1.3in} + \biggl(\int_{0}^t \int _{Q}
\widetilde{f}_1(\alpha(s)(Q),\hat
q)\overline{\gamma}_{s,t,\alpha}(\hat q) d\alpha(s)(\hat q) ds\biggr
)
\end{array}
\end{equation}
has a unique fixed point.
\end{lemma}

\begin{proof} Let $\alpha \in M(a,b)$. Now clearly from the form of \eqref{irep1} $[S\alpha](0,u,\gamma)=u$ and $[S\alpha]$ is nonnegative.   If $a$, $b$, $C_1, C_2$ are as in the hypothesis,
then
$$
\begin{array}{lll}
([S\alpha](t,u,\gamma) - \mu_{0}) &=& <\int e^{-\int
_{0}^{t} \widetilde{f}_{2}(\alpha(\tau)(Q),q)d\tau}
d(u-\mu_{0}), \cdot >\\ && \quad
 +   <\int
(e^{-\int _{0}^{t}
\widetilde{f}_{2}(\alpha(\tau)(Q),q)d\tau}-1)
 d\mu_{0}(q), \cdot >\\
 &&\quad +
\biggl(\int_{0}^t \int _{Q} \widetilde{f}_1(\alpha(s)(Q),\hat
q)\overline{\gamma}_{s,t,\alpha}(\hat q) d\alpha(s)(\hat q)\times ds
\biggr )
\end{array}$$ and

$$
\begin{array}{lll}
|[S\alpha] - \mu_{0}|_V & \leq  |u-\mu_{0}|_V +
(1-e^{-B_2b})\mu_{0}(Q)
+2b B_1C_1\\
& \leq a +(1-e^{-B_2b})\mu_{0}(Q) + 2bB_1C_1 < 2a.
\end{array}$$
We now show that $[S\alpha]$  is  continuous. This means that if
$(t_{n}, u_{n},\gamma_{n}) $ is a sequence in $
\overline{I_{b}(0)} \times (\overline{
B_{a,+}(\mu_{0})})_w \times C^{po}$ that converges to $(t, u,
\gamma )\in
 \overline{I_{b}(0)} \times
(\overline{ B_{a,+}(\mu_{0})})_w \times C^{po} $, and if $
[S\alpha]_{n}$ = $[S\alpha](t_{n};
u_{n},\gamma_{n})$ and $[S\alpha] =[S\alpha](t; u,
\gamma )$, then
  $ [S\alpha]_{n} \rightarrow
[S\alpha]$ in the weak$^*$ topology. Let

  $ Ia = <\int
e^{-\int_{0}^{t_{n}} \tilde
f_{2}(\alpha_n(\tau)(Q),q)d\tau}d(u_{n}-u)(q), \cdot>,$

$Ib= <\int[ e^{-\int_{0}^{t_{n}} \tilde
f_{2}(\alpha_{n}(\tau)(Q),q)d\tau}- e^{-\int_{0}^{t}
\tilde f_{2}(\alpha(\tau)(Q),q)d\tau}]du, \cdot>,   $

$ IIa = \biggl( \int_{t}^{t_{n}}\int_Q
\tilde{f_1}(\alpha_{n}(s)(Q),\hat q)
\overline{\gamma}_{s,t_{n},\alpha_{n}}(\hat
q)d\alpha_{n}(\hat q)ds \biggr), $

 $ IIb1 =\biggl (\int_{0}^{t}\int_Q \Bigl [
\tilde{f_1}(\alpha_{n}(s)(Q),\hat
q)-\tilde{f_1}(\alpha(s)(Q), \hat q) \Bigr ]
\overline{\gamma}_{s,t_{n},\alpha_{n}}(\hat
q)d\alpha_{n}(\hat q)ds \biggr),$

$IIb2= \biggl( \int_{0}^{t} \int_Q
\tilde{f_1}(\alpha(s)(Q),\hat
q)[\overline{\gamma}_{s,t_{n},\alpha_{n}}(\hat q) -
\overline{\gamma}_{s,t,\alpha}(\hat q) ] d\alpha_{n}(\hat
q)ds\biggr ),$ and

$IIb3= \biggl ( \int_{0}^{t} \int_Q
\tilde{f_1}(\alpha(s)(Q),\hat q)\overline{\gamma}_{s,t,\alpha}(\hat
q)
 d[\alpha_{n}- \alpha](\hat q)ds\biggl ).$

\noindent Then,  $ ([S\alpha]_{n} - [S \alpha])
 =Ia + Ib +IIa + IIb1 + IIb2 + IIb3$. We
remind the reader that the $weak^*$ topology is generated the family
of seminorms  $ \rho_f(\mu)=|\int_Q f d\mu |$, where $f \in C(Q)$.
So if $\rho_f$ is a seminorm, we need to show that
  $\rho_f([S\alpha_{n}] - [S \alpha])$ is small as
$n \to \infty $. To this end, we provide an estimate for
each of the terms above.
\begin{enumerate}
\item $\rho_f (Ia)$ is small since
$e^{-\int_{0}^{t_{n}} \tilde
f_{2}(\alpha_{n}(\tau)(Q),q)d\tau} \rightarrow
e^{-\int_{0}^{t} \tilde f_{2}(\alpha(\tau)(Q),q)d\tau}$
uniformly in $q$, $e^{-\int_{0}^{t} \tilde
f_{2}(\alpha(\tau)(Q),q)d\tau}$ is continuous in $q$
 and $u_{n} \rightarrow u $ in $ \mathcal{M}_{w} $.

\item The fact that $\rho_f(Ib)$ is small follows from the fact that $e^{-\int_{0}^{t_{n}} \tilde
f_{2}(\alpha_{n}(\tau)(Q),q)d\tau} \rightarrow
e^{-\int_{0}^{t} \tilde f_{2}(\alpha(\tau)(Q),q)d\tau}$
uniformly in $q$.
\item The fact that $\rho_f(IIb1)$ is small follows from the second estimate in Lemma \ref{E} and Theorem \ref{seminorm}.

\item Using Theorem \ref{seminorm} we get
 $$ \begin{array}{lll}\rho_f (IIb2) &\leq&
\int_{0}^t \int_Q \tilde{f_1}(\alpha(s)(Q),\hat
q)\rho_f[\overline{\gamma}_{s,t_{n},\alpha_{n}}(\hat
q) - \overline{\gamma}_{s,t,\alpha}(\hat q) ]
d|\alpha_{n}|(\hat q)ds.
\end{array}$$
Since $e^{-\int_{s}^{t_{n}} \tilde
f_{2}(\alpha_{n}(\tau)(Q),q)d\tau} \rightarrow
e^{-\int_{s}^{t} \tilde f_{2}(\alpha(\tau)(Q),q)d\tau}$ uniformly in
$(s,q)$, then $
\rho_f[\overline{\gamma}_{s,t_{n},\alpha_{n}}(\hat q)
- \overline{\gamma}_{s,t,\alpha}(\hat q) ] \rightarrow 0$  uniformly
in $(s, \hat q)$ as $n \to \infty$. Thus, our result is
immediate.
\item  For the term $IIb3$ we have
$$\begin{array}{l} \rho_f (IIb3) = \\
\quad  | \int_{0}^{t}\int_Q \tilde{f_1}(\alpha(s)(Q),\hat
q)\int_Q f(q)e^{-\int_{s}^{t} \tilde
f_{2}(\alpha(\tau)(Q),q)d\tau}d\gamma(\hat q)(q)
 d[\alpha_{n}- \alpha](\hat q)ds|. \\
 \end{array}$$

If $g_{n}(s) =\int_Q \tilde{f_1}(\alpha(s)(Q),\hat q)\int_Q
f(q)e^{-\int_{s}^{t} \tilde
f_{2}(\alpha(\tau)(Q),q)d\tau}d\gamma(\hat q)(q)
 d[\alpha_{n}- \alpha](\hat q) $, then $g_{n} \rightarrow 0$  pointwise.
 Hence our result follows by dominated convergence and the facts that $\tilde{f_1}(\alpha(s)(Q),\hat
q)\int_Q f(q)e^{-\int_{s}^{t} \tilde
f_{2}(\alpha(\tau)(Q),q)d\tau}d\gamma(\hat q)(q)$ is continuous and
$ \alpha_n \rightarrow \alpha$.

\item
By hypothesis $t_{n}\rightarrow t$, hence, the term
$\rho_f(IIa)$ is small since the integrands are bounded.
\end{enumerate}

Now for the contraction we have the following. If
$$
\begin{array}{lll}
I & = & <\int\Bigl ( e^{-\int _{0}^{t}
\tilde{f_{2}}(\alpha(\tau)(Q),q)d\tau} - e^{-\int _{0}^{t}
\tilde{f_{2}}(\beta(\tau)(Q),q)d\tau} \Bigr )du(q), \cdot>,
\\II & = &
 \biggl (\int_{0} ^t  \int_Q \widetilde{f_1}(\alpha(s)(Q),\hat
q)\overline{\gamma}_{s,t,\alpha}(\hat q) d( \alpha - \beta )(s)(\hat
q)ds\biggr ),
\\III & = &
\biggl ( \int_{0} ^t \int_{Q} \Bigl\{
\widetilde{f_1}(\alpha(s)(Q),\hat q)\overline{\gamma}_{s,t,
\alpha}(\hat q)- \widetilde{f_1}(\beta(s)(Q),\hat
q)\overline{\gamma}_{s,t,\beta}(\hat q ) \Bigr \} d\beta(s)(\hat q)
ds\biggr ),
\end{array}
$$
then $([S\alpha] -[S\beta]) = I + II + III$, and  $|[S\alpha]
-[S\beta]|_V$ $\leq |I|_V + |II|_V + |III|_V$ $ \leq (2bL_2 C_1  +
2bB_1 + 2bC_1C_2) \|\alpha - \beta \|_S .$ Hence, $S$ is a
contraction mapping. Therefore, $S$ has a unique fixed point in
$M(a,b)$.
\end{proof}

We will denote this fixed point by $\widetilde{\varphi}_{a}.$

\begin{proposition}\label{LS}(Local Solution) If $  \mu_{0} \in  \mathcal{M}_+$, and $b$ is as in Lemma
\ref{FP}, then
\begin{description}
\item 1. the function $\widetilde \varphi_a$ satisfies
\begin {equation} \label{IREP}
 \begin{array}{l} {\widetilde{\varphi}}_{a}(t;u,\gamma)=<\int e^{-\int
_{0}^{t} \tilde
{f_{2}}(\widetilde{\varphi}_{a}(\tau)(Q),q)d\tau}du(q), \cdot> \\
\hspace{0.8 in} + \left (\int _{[0, t]\times Q}
\widetilde{f}_1(\widetilde{\varphi}_{a}(s;u,\gamma)(Q),\hat
q)\overline{\gamma}_{s,t,\widetilde{ \varphi}}
d\widetilde{\varphi}_{a}(s)(\hat q)\times ds \right)
\end{array} \end{equation}
\item and is a local solution to
\begin {equation}\label{SOL} \left\{
\begin{array}{l}
  \dot{x}(t) = \widetilde{F} (x(t), \gamma) \\
  \quad =
\biggl (\int_Q \widetilde{f_1}(x(t)(Q),\hat q)\gamma (\hat q)
dx(t)(\hat q) \biggr )- <\int \widetilde{f_{2}}(x(t)(Q),\hat
q)dx(t) (\hat q),\cdot>
  \\
x(0)=u.
\end{array}
\right.
\end{equation}
\item 2. $\widetilde{\varphi}_{a}$ is nonnegative and continuous.
\end{description}
\end{proposition}
\begin{proof}
\begin{enumerate}
\item
We differentiate the integral
representation \eqref{IREP} and show that it satisfies \eqref{SOL}.
Then we use uniqueness of solution given that we have Lipschitzicity
by Lemma \ref{LF}. If $\widetilde{\varphi}_a =\mu_1 + \mu_2$, then
$\dot{\widetilde{\varphi}}_a =\dot{\mu_1} + \dot{\mu_2}$, where $$
\begin{array}{lll}  \mu_1 &=&<\int e^{-\int _{0}^{t} \tilde
{f_{2}}(\widetilde{\varphi}_{a}(\tau)(Q),q)d\tau}du, \cdot>
\end{array}$$
and
$$\begin{array}{lll}
 \mu_2&=& \int_{[0, t]\times Q}
\widetilde{f}_1(\widetilde{\varphi}_{a}(s)(Q),\hat q
)\overline{\gamma}_{s,t, \widetilde{\varphi}}(\hat q)
d\widetilde{\varphi}_{a}(s)(\hat q)\times ds . \end{array}$$

Clearly $$ \dot{\mu}_1 = <\int
-\widetilde{f}_2(\widetilde{\varphi}_a(t)(Q),q) d\mu_1(q), \cdot>.$$  Since
$${\mu_2}(f) = \int_{0}^t \biggl[ \int_Q
\widetilde{f}_1(\widetilde{\varphi}_a(s)(Q), \hat q)\int_Q
f(q)e^{-\int _{s}^{t} \tilde
{f_{2}}(\widetilde{\varphi}_{a}(\tau)(Q),q)d\tau} d\gamma(\hat
q)(q)d\widetilde{\varphi}_a\biggr] ds,$$

then $$ \begin{array}{lll} \dot{\mu}_2(f) &=& \int_{0}^t\biggl
[ \int_Q \widetilde{f}_1(\widetilde{\varphi}_a(s)(Q), \hat q)\int_Q
f(q) (-
\widetilde{f}_2(\widetilde{\varphi}_a(t)(Q),q))d\gamma_{s,t,\widetilde{\varphi}_a}(\hat
q)(q) d\widetilde{\varphi}_a \biggr ] ds\\
&& +
 \int_Q \widetilde{f}_1( \widetilde{\varphi}_a(t)(Q),\hat q)\int_Q
f(q)d\gamma(\hat q)(q) d\widetilde{\varphi}_a(t)(\hat q)\\
& =& \mu_2 (-\widetilde{f}_2(\widetilde{\varphi}_a(t)(Q), \cdot)f ) +
\biggl (\int_Q \widetilde{f}_1( \widetilde{\varphi}_a(t)(Q),\hat
q)\gamma(\hat q)
d\widetilde{\varphi}_a(t)(\hat q) \biggl )(f)\\
 & =& \hspace{-.1in}<\int- \widetilde{f}_2(\widetilde{\varphi}_a(t)(Q),\cdot)d\mu_2,f> +
\biggl (\int_Q \widetilde{f}_1( \widetilde{\varphi}_a(t)(Q),\hat
q)\gamma(\hat q) d\widetilde{\varphi}_a(t)(\hat q) \biggl
)(f).\end{array}$$

Hence, $$ \dot{\widetilde{\varphi}_a}= \biggl (\int_Q
\widetilde{f_1}(\widetilde{\varphi}_a(t)(Q),\hat q)\gamma (\hat q)
d\widetilde{\varphi}_a(t)(\hat q) \biggr )- <\int
\widetilde{f_{2}}(\widetilde{\varphi}_a(t)(Q),q
)d\widetilde{\varphi}_a(t), \cdot>.$$

\item This follows from Lemma \ref{FP}.
\end{enumerate}
\end{proof}

For fixed $u, \gamma $ we denote this local solution to \eqref{SOL}
by $ \widetilde{\mu}_a$, i.e., $\widetilde{\mu}_{a}(t)=\widetilde\varphi _a
(t;u,\gamma)$. Since $\widetilde{\varphi}_{a}$ is nonnegative, we
see that $\widetilde{\mu}_{a} $ is nonnegative.

\subsubsection{Proof of Theorem \ref{main}}

Let $ a>0$, by Proposition \ref{LS} we see that the dynamical
system, $\widetilde{\varphi}_a$, exists on a small interval
$I_b(0)$ . Since $\widetilde{\varphi}_a \in B_{2a}( \mu_{0})$,
then $\widetilde{\varphi}_{a}( t;u, \gamma)(Q)  < \widetilde{K}$ and $\widetilde{F}(\widetilde{\varphi}_{a}( t;u, \gamma),\gamma)=  F(\widetilde{\varphi}_{a}( t;u, \gamma),\gamma)$ on $I_b(0)$. Hence, equation
\eqref{M} has the local solution $\widetilde{\varphi}_a$ on $I_b(0)$. This means by Lemma \ref{LS} that for fixed $u$ and $\gamma$,
 $\widetilde{\varphi}_a(\cdot,u,\gamma): I_b(0) \to
{\cal{M}}_{V,+}$ is continuously differentiable and satisfies \eqref{M}. We will denote this solution as
$\mu_a$ and the dynamical system as $\varphi_a$.

Moreover,  from the nonnegativity of the local solution to \eqref{M}, $\mu_a$, and the nonincreasing property of $f_1$ with respect to $X$ given in assumption (A1), it is easy to show that this solution satisfies
$\dot \mu_a (t)(Q) \le M_{f_1} \mu_{a}(t) (Q)$, where $M_{f_1}=\max_{q \in Q} f_1(0,q)$. Hence, if we let $g(t,s) =M_{f_1}s$, then using Theorem \ref{GlobalEx} we see that $\mu_a $ can
be extended to all of $\mathbb{R}_+$.

Hence $\mu_a$ is a nonnegative global solution to \eqref{M} for
initial measures in a variation bounded set. On any interval $J$, if
$\mu_a $ is a solution to \eqref{M}, then the set $\{ \mu_a(t)  : t
\in J \}$ is a bounded set in total variation. Hence we can use
Lemma \ref{LF} along with the Gronwall inequality to show that this
solution is unique.

Since $\vec{\textbf{0}} \in \mathcal{M}_+ $,
and $\mathbb{R}_+ \times (\mathcal{M}_+)_w \times C^{po}$ = $
\bigcup_{N \in \mathbb{Z}_+}  \mathbb{R}_+ \times
(\overline{B_{N,+}(\vec{\textbf{0}})})_w \times C^{po}, $ we let
$\varphi = \bigcup_{N \in \mathbb{Z}_+} \varphi_N $ and Theorem
\ref{main} is immediate.

\section{ \bf Reduction to Special Cases}\label{Reduction}
 Selection
and mutation models have been considered on discrete
 strategy/trait spaces \cite{AA1,AzmyShu,BT} and continuous strategy/trait
spaces \cite{calsina,CALCAD,SES}. In this section we demonstrate the unifying power of the measure theoretic formulation.  In particular, we present the
correct choices of initial measure $u$ and the selection-mutation
kernel $\gamma (\hat q)$ such that the model \eqref{M}  reduces to
each of the cases of interest. Given that our model is nonnegative, we can use Theorem \ref{Pos} and write our model using set function notation.
\label{Reduction}
\begin{enumerate}
\item {\it Reduction to pure selection model}: Let $\gamma(\hat q) =\delta_{\hat q}$ and $u\in
\mathcal{M}_+.$ Substituting these parameters in \eqref{M} one
obtains the pure selection model
\begin {equation} \left\{\begin{array}{ll}\label{selection}
 \displaystyle \frac{d}{dt}{\mu}(t;u, \gamma)(E) = \int_E \left ({f}_1(\mu(t)(Q), \hat q)-{f}_{2}(\mu(t)(Q),\hat q)
\right )d\mu(t)(\hat q) \\
\mu(0;u,\gamma)=u.
\end{array}\right.\end{equation}

\item {\it Reduction to density model}:  Let $Q \subset {\rm int}(\mathbb{R}_+^n)$ and $\gamma(\hat q), u \in L_1(Q, \nu)$, i.e, both are absolutely
continuous with respect to a measure $\nu$.
 If $d\gamma(\hat q)= P(q, \hat q) d\nu(q)$ and $ du= c_u(q) d\nu(q) $, then substituting these
expressions into  \eqref{IREP} and using Fubini's theorem we see
that there exists $c_0(t,q), f(t, q) \in L^1(Q, \nu)$ such that
$d\varphi(t;u,\gamma) = c_0(t,q)d\nu(q) + f(t,q)d\nu(q)= (c_0(t,q)+
f(t,q))d\nu(q)$. Hence if $\nu =dq$, there exists $x_{u,\gamma}(t,q)
\in L^1( \mathbb{R}_+\times Q, dq) _+$ such that $d\mu(t)
=x_{u,\gamma}(t,q) dq.$ By Fubini's theorem, equation \eqref{M}
becomes
$$\begin{array}{lll} \dot{\mu}(E) &= &\int_E
\dot{x}_{u,\gamma}(t,q)dq\\
& = &  \hspace{-0.1in}\int_E \Bigl [ \int_Q f_1(\mu(t) (Q), \hat
q)P(q, \hat q)x_{u,\gamma}(t,\hat q)d \hat q -f_2( \mu(t)(Q), q)
x_{u,\gamma}(t, q)\Bigr ]dq,
\end{array}$$ for all $E \in \mathcal{B}(Q).$
 Hence
\begin {equation} \left\{
\begin{array}{ll}
\dot{x}_{u,\gamma}(t,q) = \int_Q f_1(\mu(t) (Q), \hat q)P(q, \hat
q)x_ {u,\gamma}(t,\hat q)d \hat q -f_2(
\mu(t)(Q), q) x_{u,\gamma}( t,q)\\
x_{u,\gamma}(0, q)=x_u( q). \end{array}\right.
\end{equation}
This is the density replicator-mutator model \eqref{dsmm}.

\item {\it Reduction to discrete model}: Assume that
$ \gamma, u $ are both discrete, i.e., their support is countable
and consists of isolated points. For $\gamma$ this means that there
is a discrete set $\Lambda$ which contains the support of
$\gamma(\hat q)$ for all $\hat q$. Assume there exists a sectionwise
continuous function $P(q,\hat q)$, and a family of measures $\nu(q)$
all having the same discrete support $\Lambda$ such that
$d\gamma(\hat q)= P(q, \hat q) d\nu(q)$. Then if we substitute these
expressions into equation \eqref{IREP} and use Fubini's theorem we
see that there exists $c_u(t,q), c_\nu(t, q) \in C(\mathbb{R}_+\times
Q) _+$ such that $d\varphi(t;u,\gamma) = c_u(t,q)du(q) +
c_\nu(t,q)d\nu(q)$. Hence, equation \eqref{M} becomes
\begin {equation} \left\{
\begin{array}{ll}
\dot{\mu}= \int_E \Bigl [ \int_{Q} f_{1}(\mu(t)(Q),\hat q) P(q,\hat
q)c_u(t,\hat q) du(\hat q)
\\ \qquad +  \int_Q f_{1}(\mu(Q),\hat q)P(q, \hat q)c_{\nu}(t,\hat q)
d\nu(\hat q) \Bigr ]d\nu (q)
  \\ \qquad -\int_E f_{2}(\mu(Q), q)c_{\nu}(t, q) d\nu(q) -\int_E f_{2}(\mu(Q), q)c_u(t, q)
  du(q) \\
\mu(0)=u .
\end{array}
\right.
\end{equation}

 If  $E =\{q_i \} $ and $\text{supp} (\nu)$ denotes the support of the $\nu$,
 the above becomes a discrete system given by
\begin {equation}\label{discrete1} \left\{
\begin{array}{l}
\dot{\mu}(\{q_i\})= \sum _{\hat q_j \in \text{supp}{\nu}}
f_{1}(\mu(t)(Q), \hat q_j)P(q_i,\hat q_j)[ c_u(t, \hat q_j)+
c_\nu(t, \hat q_j] \\
\qquad  \quad - f_{2}(\mu(t)(Q),  q_i)[c_u(t, q_i) + c_\nu(t, q_i)]
  \\
\mu(0)(\{q_i\})= c_u(0,q_i) .
\end{array}
\right.
\end{equation}

For example, if $N=2,...$ $ \text{supp} (u)  =\{ q_{i}\}_{i=1}^{N}$,
$q= (a(q),b(q))$, $x_{i} =\mu(\{q_{i}\})$,
$f_{1}(X,q)=a(q),f_{2}(X,q)=b(q)X, \mu(Q)=X$, then through a
suitable change of variable ($ y_i= a_ix_i $) equation
\eqref{discrete1} reduces to the exact differential equation system
studied in \cite{AzmyShu}.
\item Many authors in EGT theory assume that $f_1$ is a fitness
function, $f_2(X,q) =\int_Q f_1(X,q)d\mu$ is an average fitness, and
$\mu$ is a probability measure. The models are mostly on
$\mathbb{R}_+^n$ and the n-simplex is invariant. From our assumptions
we can incorporate a version of this also by using the companion IVP
to \eqref{M},
 $$P( t,u, \gamma)(E) =
\frac{{\mu}(t)(E)} {{\mu}(t)(Q)}.$$  It has the dynamics
\begin{equation} \label{MM}
\begin{array}{l}
\frac{d}{dt}{P}(t; u, \gamma)(E)=  \int_{Q} \Bigl
[f_1(\mu(t)(Q), \hat q) \gamma( \hat q) (E) \\
\hspace{1.6in} -
\Bigl(\int_Q f_1(\mu(t)(Q), q)dP(t)\Bigr ) P(t)(E)\Bigr ] dP(t)(\hat q) \\
 \hspace{1.4in} - \int_E [f_2(\mu(t)(Q), \hat q) - \int_Q f_2(\mu(t)(Q), q) dP(t)]
dP(t)(\hat q).
\end{array}\end{equation}
\item {\it Reduction to Density Dependent Replicator Equation}:
Define $$f(\mu(Q),q) = f_1(\mu(Q),q)-f_2(\mu(Q),q).$$ If
$\pi(dq,\mu) = F(\mu)(dq) = f(\mu(Q),q)\mu(dq) $, then \eqref{MM}
becomes
$$ \dot{P}(t)(E)=\int_E [f(X,q)-\overline{f}(X,q)]dP(q), $$ where $\overline{ f}=\int_Q f
d\mu$. This is exactly the density dependent Replicator equation.

\item {\it Reduction to Density Dependent Quasi-species Equation}:
Likewise if we interpret $f_1(\mu(Q),q)$ as a net fitness and
$f_2(\mu)= \int_Q f_1(\mu(Q),q)d\mu$ as the average fitness, then
\eqref{M} becomes the Density Dependent Replicator-Mutator equation.

\end{enumerate}

\section{ \bf Concluding Remarks}We have formulated a density dependent EGT (selection-mutation)
 model on the space of measures and provided a framework which is rich
 enough to allow pure selection, selection-mutation, and discrete and continuous strategy spaces,
 all under one setting.
 We also established the well-posedness of this EGT model.

There are several future paths to take from this point. We will
mention one application and one  mathematical future pathway.
Modeling tumor growth, cancer therapy and viral evolution are
immediate applications. For example, tumor heterogeneity is one main
cause of tumor robustness. Tumors are robust in the sense that
tumors are systems that tend to maintain stable functioning despite
various perturbations. While tumor heterogeneity describes the
existence of distinct subpopulations of tumor cells with specific
characteristics within a single neoplasm. The mutation between the
subpopulations is one major factor that makes the tumor robust. To
date there is no unifying framework in mathematical modeling of
carcinogenesis that would account for parametric heterogeneity
\cite{karev}. To introduce distributed parameters (heterogeneity)
and mutation is essential as we know that cancer recurrence, tumor
dormancy and other dynamics can appear in heterogeneous settings and
not in homogeneous settings. Increasing technological sophistication
has led to a resurgence of using oncolytic viruses in cancer
therapy. So in formulating a cancer therapy it is useful to know
that in principle
 \emph{a heterogeneous oncolytic virus must be used to eradicate a tumor cell.}

 One mathematical future path is to perform asymptotic analysis on the model. There are two essential things that need to be addressed if we wish to be able to perform asymptotic analysis of our model. First, we need a state space with the  property that if the
  measure valued dynamical system has an initial condition as a finite signed Borel measure then the asymptotic
  limits will also be in this space.  The second problem is that often there will be more than one strategy of a given fitness.
  In \eqref{logiseq}, a Dirac mass emerged as it is assumed that only a unique fittest class exists.
  In reality, this may not be the case and more than one fittest class can exist. In particular, it is possible that a
   continuum of fittest strategies exist (see Figure 1 for an example).  So our mathematical structure must include the ability
   to demonstrate the convergence of the model solution to a measure supported on a continuum of strategies.

  These two difficulties  coupled with our desire to study the problem of parameter estimation in these models imply that some
   form a ``weak" or ``generalized"  asymptotic limit must be formulated.  These weak limits need to live in
   a certain ``completion" of the space of finite signed measures. We will explore these topics in a forthcoming study.

\begin{center}
{\bf Insert Figure 1 Here}
\end{center}


 \vspace{0.0 in}
 \noindent {\bf Acknowledgements:} The authors would like to thank Horst
 Thieme for thorough reading of an earlier version of this
 manuscript and for the many useful comments. The authors would also like to thank their colleague Ping Ng
 for helpful discussions. This work was partially supported by the National Science Foundation under grant \# DMS-0718465.

\section{ \bf Appendix}

 For the convenience of the reader, we next state a few known results that are used in our analysis.
\begin{theorem}\label{DUI}(Differentiation Under Integral)
If $ a\leqslant \alpha \leqslant b$, let
$$\phi(\alpha)=\int_{u_{1}(\alpha)}^{u_{2}(\alpha)}f(x,\alpha)dx. $$
Then
$\phi_{\alpha}=\int_{u_{1}(\alpha)}^{u_{2}(\alpha)}f_{\alpha}(x,\alpha)dx
+f(u_{2},\alpha)u_{2,\alpha}-f(u_{1},\alpha)u_{1,\alpha}$ provided
$f, f_{\alpha}$ are continuous in some region containing $ \{x|
u_{1} \leqslant x \leqslant u_{2}\}\times (a,b)$ and $ u_{1}, u_{2}
\in C^{1}(a,b) $.
\end{theorem}

\begin{definition} \label{duality}\cite[pg. 151]{AliBord}
 A dual pair or a dual system is a pair $<X, X'>$ of vector spaces over a field $F$
together with a function $(x, x') \mapsto <x,x'> \in ~F$ satisfying
the following:

\begin{enumerate}
\item The map $x \mapsto <x,x'> $ is linear for each $x'$.
\item If $<x,x'>=0$ for each $x'$, then $x=0$.
\item The map $x' \mapsto <x,x'> $ is linear for each $x$.
\item If $<x,x'>=0$ for each $x$, then $x'=0$.
\end{enumerate}
\end{definition}
Each space of a dual pair $<X,X'> $ can be interpreted as a set of
linear functionals on the other. For instance, each $x \in X$
defines the linear functional $ x'\mapsto <x,x'>$. If $A\subseteq
X$, then it is called $X' - bounded$ if $sup_{x \in A}|<x,x'>| $ is
bounded for every $x' \in X'$.  For each $X'-bounded$ subset
$A\subseteq X$ we define the semi-norm  on $X'$ $$ p_A(x')=\sup_{x
\in A}|<x,x'>|.
$$

If $\ss$ is a system of $X'-bounded$ subsets, the family $\{p_A | A
\in \ss \}$ generates a Hausdorf locally convex topology called the
$\ss$ topology. A net $(x'_{\alpha})$ converges to $x'$ iff $p_{A}(
x_{\alpha}'-x') \rightarrow 0 $ for all $ A \in \ss$. If $\ss$
consists of singletons, then it is called the weak* topology on $X'$
and is often also denoted as $\sigma(X',X)$.

\begin{theorem}\label{Duality}\cite[pg. 153]{AliBord}
(Duality pairs are weakly dual) Let $X,Y$ be topological vector
spaces over a field $F$ forming a dual pair. The topological dual of $( X, \sigma(X,Y))$
is $Y$. Similarly $(Y, \sigma(Y,X)'= X$. Here $\sigma(X,Y)$ ( $
\sigma(Y,X)$), denotes the weak topology on $X$ generated by the
family of linear functionals $\{<\cdot, y>\}_{y \in Y}$. Also $X'$
is the notation used for the continuous dual of $X$.
\end{theorem}

 The next theorem is concerned with
\begin{equation}\label{GLEQ}
 x'=f(t,x),~~ x(t_0) = x_0, \end{equation} where $f \in C[\mathbb{R}_+ \times E,
E]$, $E$ being a Banach space.
\begin{theorem}\label{GlobalEx} \cite[pg. 145]{LakLeela} Assume that $$\|f(t,x) \| \leq
g(t,\|x\|), ~~ (t,x) \in \mathbb{R}_+ \times E ,$$ where $ g \in
C[\mathbb{R}_+\times\mathbb{R}_+, \mathbb{R}_+]$, $g(t,u)$ is
nondecreasing in $u$ for each $t \in \mathbb{R}_+ $, and the maximal
solution $r(t,t_0,u_0)$ of the scalar differential equation
$$ u'=g(t,u),~~ u(t_0)=u_0 \geq 0 ,$$
exists on $[t_0, \infty).$ Suppose that f is smooth enough to assure
local existence of solutions to \eqref{GLEQ} for any $(t_0,x_0) \in
\mathbb{R}_+\times E.$  Then the largest interval of existence of
any solution $x(t,t_0,x_0)$ of \eqref{GLEQ} such that $\|x_0 \| \leq
u_0$ is $[t_0, \infty)$. If in addition $r(t,t_0,u_0)$ is bounded,
then $ \lim_{t \to \infty} x(t,t_0,x_0) =y \in E. $
\end{theorem}

\begin{definition}\label{integral} \cite[ III.33]{NB2} Let $X$ be locally compact, $E$ a
Hausdorff
locally convex space, and $\mu$ a measure on the Borel sets of $X$.
For every $f \in C_c(X;E)$ we call the integral of $f$ with respect
to $\mu$, $\int f d\mu $, the element of ${E'}^\sharp$ where ${E'}$
is the continuous dual and ${E'}^\sharp$ is the algebraic dual
defined by

$$ \biggl < \int f d\mu, {z'} \biggr > = \int_X <f(x),{z'}> d\mu(x), ~~~~~~~~~~~~ for~ all~ {z'} \in  {E'}. $$
\end{definition}

\begin{theorem}\label{seminorm}\cite[III.37]{NB2} Let $X$ be as in
Definition \ref{integral}, and let $\mathcal{B}_X$ denote the Borel
sets on $X$. Suppose $f$ is a continuous mapping with compact
support of $(X, \mathcal{B}_X)$ into a Hausdorf locally convex space
$E$ and $q$ is a continuous semi-norm on $E$. Then for every measure
$\mu$ on $(X, \mathcal{B}_X)$ such that $\int f d\mu \in E,$ $$ q
\biggl( \int f d\mu \biggr) \leq \int(q\circ f) d|\mu |.
$$

\end{theorem}

\begin{theorem}\label{int}\cite[III.37]{NB2} Let $X$ be as in
Definition \ref{integral}, let $E$ be a Hausdorf locally convex
space, and $f \in C_c(X;E)$. If $f(X)$ is contained in a complete
convex subset $A$ of $E$, then $\int f d\mu \in E$. \end{theorem}

\begin{figure}[htbp] \label{Fig1}
 \resizebox{12cm}{!}{\includegraphics{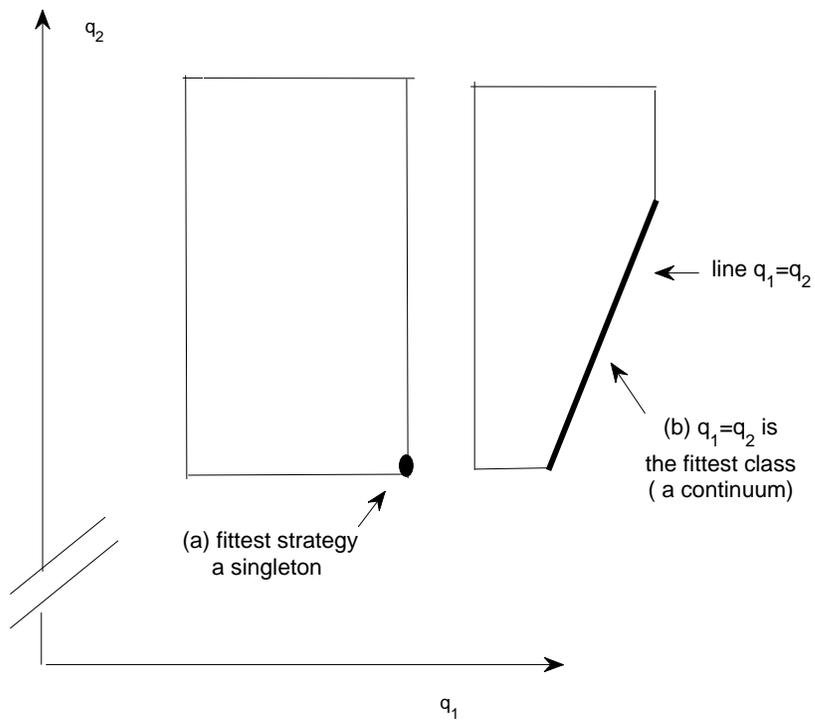}}
 \caption{Two examples of strategy spaces.}
\label{strategyspace}
\end{figure}

\end{document}